\documentclass[11pt,a4paper]{article}

\usepackage{amsmath}
\usepackage{amsthm}
\usepackage{amssymb}
\usepackage{amscd}
\usepackage[all]{xy}

\title{Vanishing Theorems on Toric Varieties\footnote{This paper was partially supported
by the National Natural Science Foundation of China (Grant No.\ 11422101).}}
\author{Qihong Xie}
\date{}
\theoremstyle{plain}
\newtheorem{prop}{Proposition}[section]
\newtheorem{lem}[prop]{Lemma}
\newtheorem{thm}[prop]{Theorem}
\newtheorem{cor}[prop]{Corollary}

\theoremstyle{definition}
\newtheorem{defn}[prop]{Definition}
\newtheorem*{ack}{Acknowledgments}
\newtheorem*{nota}{Notation}

\theoremstyle{remark}
\newtheorem{rem}[prop]{Remark}
\newtheorem{ex}[prop]{Example}

\newcommand{\Q}{\mathbb Q}
\newcommand{\R}{\mathbb R}

\newcommand{\Z}{\mathbb Z}

\newcommand{\PP}{\mathbb P}
\newcommand{\OO}{\mathcal O}

\newcommand{\HH}{\mathcal H}
\newcommand{\LL}{\mathcal L}
\newcommand{\GG}{\mathcal G}
\newcommand{\VV}{\mathcal V}

\newcommand{\FF}{\mathcal F}

\newcommand{\CC}{\mathcal C}

\newcommand{\UU}{\mathcal U}
\newcommand{\XX}{\mathcal X}
\newcommand{\YY}{\mathcal Y}
\newcommand{\WW}{\mathcal W}
\newcommand{\DD}{\mathcal D}

\newcommand{\Hom}{\mathop{\rm Hom}\nolimits}

\newcommand{\spec}{\mathop{\rm Spec}\nolimits}

\newcommand{\divisor}{\mathop{\rm div}\nolimits}

\newcommand{\codim}{\mathop{\rm codim}\nolimits}
\newcommand{\ra}{\rightarrow}

\newcommand{\wt}{\widetilde}

\newcommand{\inj}{\hookrightarrow}

\begin{document}

\maketitle

\begin{abstract}
In this paper, we clarify the mistakes made in the former article entitled
``Vanishing Theorems on Toric Varieties in Positive Characteristic''.
On the other hand, we use the positive characteristic method to reprove
the Bott vanishing theorem, the degeneration of the Hodge to de Rham
spectral sequence and the Kawamata-Viehweg vanishing theorem for log pairs
on toric varieties over a field of arbitrary characteristic,
where concerned Weil divisors are torus invariant.
\end{abstract}

\setcounter{section}{0}
\section{Introduction}\label{S1}

Throughout this paper, we always work over a field $F$, which is either
{\it a field $K$ of characteristic zero or a perfect field $k$ of characteristic $p>0$}.
The main purpose of this paper is to use the positive characteristic method
to reprove various vanishing theorems on toric varieties.
See Definition \ref{2.2} for the definition of $\wt{\Omega}^\bullet_X(\log D)$,
the Zariski-de Rham complex of $X$ with logarithmic poles along $D$.
The following are the main theorems in this paper.

\begin{thm}[Bott vanishing]\label{1.1}
Let $X$ be a projective toric variety over $F$, $D$ a reduced torus invariant
Weil divisor on $X$, and $L$ an ample invertible sheaf on $X$.
Then for any $j>0$ and any $i\geq 0$,
\[H^j(X,\wt{\Omega}^i_X(\log D)\otimes L)=0.\]
\end{thm}

\begin{thm}[Degeneration of Hodge to de Rham spectral sequence]\label{1.2}
Let $X$ be a proper toric variety over $F$, and $D$ a reduced torus invariant
Weil divisor on $X$. Then the Hodge to de Rham spectral sequence degenerates
at $E_1$:
\begin{eqnarray*}
E_1^{ij}=H^j(X,\wt{\Omega}^i_X(\log D))\Longrightarrow
\mathbf{H}^{i+j}(X,\wt{\Omega}^\bullet_X(\log D)).
\end{eqnarray*}
\end{thm}

\begin{thm}[Kawamata-Viehweg vanishing]\label{1.3}
Let $X$ be a projective toric variety over $F$, and $H$ a nef and big
torus invariant $\Q$-divisor on $X$. Then for any $i>0$,
\[H^i(X,K_X+\ulcorner H\urcorner)=0.\]
\end{thm}

In fact, the above results have been initiated by Danilov \cite{da},
Buch-Thomsen-Lauritzen-Mehta \cite{btlm}, Musta\c{t}\v{a} \cite{mu},
and have already been proven by Fujino \cite{fu07} via the multiplication maps.
In this paper, we shall reprove these results by means of the positive
characteristic method, which consists of two steps: the first one is
to prove the positive characteristic part by using the decomposition of
de Rham complexes; the second one is to deduce the characteristic zero part
by means of the reduction modulo $p$ method.

In \S \ref{S2}, we will clarify the mistakes made in a former paper.
In \S \ref{S3}, we will recall some definitions and preliminary results.
\S \ref{S4} and \S \ref{S5} are devoted to the proofs of the positive
characteristic part and the characteristic zero part of the main theorems
respectively. For the necessary notions and results in toric geometry,
we refer the reader to \cite{da}, \cite{od}, \cite{fu93} and \cite{clh}.

\begin{nota}
We use $[B]=\sum [b_i] B_i$ (resp.\ $\ulcorner B\urcorner=\sum \ulcorner
b_i\urcorner B_i$, $\langle B\rangle=\sum \langle b_i\rangle B_i$)
to denote the round-down (resp.\ round-up, fractional part)
of a $\Q$-divisor $B=\sum b_iB_i$, where for a real number $b$,
$[b]:=\max\{ n\in\Z \,|\,n\leq b \}$, $\ulcorner b\urcorner:=-[-b]$
and $\langle b\rangle:=b-[b]$.
\end{nota}

\begin{ack}
I would like to express my gratitude to Professors Luc Illusie and
Olivier Debarre for useful comments.
\end{ack}

\section{Erratum to a former paper}\label{S2}

In the former paper \cite{xie} entitled ``Vanishing Theorems on Toric Varieties
in Positive Characteristic'', the author proved that Theorems \ref{1.1}--\ref{1.3}
hold for not necessarily torus invariant divisors on toric varieties in positive
characteristic. Recently, an anonymous algebraic geometer provides the following
example to show that Theorems \ref{1.1} does not hold for not necessarily torus
invariant divisors, which means that some mistakes were contained in the arguments
of \cite{xie}.

\begin{ex}\label{5.1}
Let $X=\PP^2$, $D\subseteq X$ a smooth projective curve of degree $d\geq 4$,
and $L=\OO_X(1)$. Then the cohomology exact sequence associated to
the short exact sequence
$$0\ra \Omega_X^1\otimes L\ra \Omega_X^1(\log D)\otimes L\ra \OO_D\otimes L\ra 0$$
together with the standard Bott vanishing for $\PP^2$ show that
$$H^1(X,\Omega_X^1(\log D)\otimes L)\cong H^1(X,L|_D).$$
Then the cohomology exact sequence associated to the short exact sequence
$$0\ra \OO_X(-d+1)\ra \OO_X(1)\ra \OO_D(1)\ra 0$$
shows that $H^1(X,L|_D)\cong H^2(X,\OO_X(-d+1))$, which is nonzero as $d\geq 4$.
Therefore $H^1(X,\Omega_X^1(\log D)\otimes L)\neq 0$, which implies that Theorem \ref{1.1}
does not hold for not necessarily torus invariant divisors on toric varieties
in arbitrary characteristic.
\end{ex}

The mistakes stem from the following ambiguity: in Definition 2.3 of \cite{xie},
a lifting of the relative Frobenius morphism $\wt{F}: \wt{X}\ra \wt{X}'$ is said to
be compatible with $\wt{D}$ if $\wt{F}^*\OO_{\wt{X}'}(-\wt{D}')=\OO_{\wt{X}}(-p\wt{D})$,
which means the equality of sheaves of ideals, but not an isomorphism of line bundles.
However, in the proof of Theorem 3.7 (ii) of \cite{xie}, the author only proved that
there exists an isomorphism of line bundles for not necessarily torus invariant divisors,
which is not enough for obtaining applications to not necessarily torus invariant divisors
on toric varieties.

\begin{rem}\label{5.2}
It should be mentioned that, in \cite{xie}, Theorem 2.4--Theorem 2.8 and Theorem 3.1--Theorem 3.6
are correct since these results have been proven in more general situations; whereas Theorem 3.7 (ii)
is valid for torus invariant divisors on toric varieties, therefore Theorems 1.1--1.3 of
\cite{xie} do hold true for torus invariant divisors on toric varieties in positive characteristic.
\end{rem}

\section{Preliminaries}\label{S3}

First of all, we recall some definitions and notations for toric varieties
from \cite{fu93}. Let $N$ be a lattice of rank $n$, and $M$ the dual lattice
of $N$. Let $\Delta$ be a fan consisting of strongly convex rational
polyhedral cones in $N_\R$, and let $A$ be a ring. In general, we denote
the toric variety associated to the fan $\Delta$ over the ground ring $A$
by $X(\Delta,A)$. More precisely, to each cone $\sigma$ in $\Delta$, there
is an associated affine toric variety $U_{(\sigma,A)}=\spec A[\sigma^\vee
\cap M]$, and these $U_{(\sigma,A)}$ can be glued together to form the toric
variety $X(\Delta,A)$ over $\spec A$.

Let $X$ be a normal variety over $F$, and $D$ a reduced Weil divisor on $X$.
Then there exists an open subset $U$ of $X$ such that $\codim_X(X-U)\geq 2$,
$U$ is smooth over $F$, and $D|_U$ is simple normal crossing on $U$.
Such an $U$ is called a required open subset for the log pair $(X,D)$.

\begin{ex}\label{2.1}
(i) Let $X$ be a smooth variety over $F$, and $D$ a simple normal crossing
divisor on $X$. Then we can take the largest required open subset $U=X$.

(ii) Let $X=X(\Delta,F)$ be a toric variety, and $D$ a torus invariant reduced
Weil divisor on $X$. Take $U$ to be $X(\Delta_1,F)$, where $\Delta_1$ is the fan
consisting of all 1-dimensional cones in $\Delta$. Then it is easy to show that
$U$ is an open subset of $X$ with $\codim_X(X-U)\geq 2$, $U$ is smooth over $F$,
and $D|_U$ is simple normal crossing on $U$. Hence $U$ is a required open subset
for $(X,D)$.
\end{ex}

\begin{defn}\label{2.2}
Let $X$ be a normal variety over $F$, and $D$ a reduced Weil divisor on $X$.
Take a required open subset $U$ for $(X,D)$, and denote $\iota:U\hookrightarrow X$
to be the open immersion. For any $i\geq 0$, define the Zariski sheaf of differential
$i$-forms of $X$ with logarithmic poles along $D$ by
\begin{eqnarray*}
\wt{\Omega}^i_X(\log D)=\iota_*\Omega^i_U(\log D|_U).
\end{eqnarray*}
Since $\codim_X(X-U)\geq 2$ and $\Omega^i_U(\log D|_U)$ is a locally free
sheaf on $U$, $\wt{\Omega}^i_X(\log D)$ is a reflexive sheaf on $X$, and
the definition of $\wt{\Omega}^i_X(\log D)$ is independent of the choice
of the open subset $U$. Furthermore, we have the Zariski-de Rham complex
$(\wt{\Omega}^\bullet_X(\log D),d)$.
\end{defn}

It follows from \cite[Proposition II.8.10]{ha77} that the K\"ahler sheaves of
relative differentials are compatible with base changes. Similarly, we can show
that the Zariski sheaves are compatible with base changes.

\begin{defn}\label{2.3}
Let $f:\XX\ra S$ be a separated morphism of varieties over $F$. Assume that
$\UU$ is an open subset of $\XX$ and $\DD$ is a reduced Weil divisor on $\XX$
such that $\codim_{\XX/S}(\XX-\UU)\geq 2$, $\UU$ is smooth over $S$, and $\DD|_{\UU}$
is relatively simple normal crossing over $S$. Denote $\iota:\UU\hookrightarrow\XX$
to be the open immersion, and for any $i\geq 0$, define the Zariski sheaf of relative
differentials of $\XX$ over $S$ by
\begin{eqnarray*}
\wt{\Omega}^i_{\XX/S}(\log\DD)=\iota_*\Omega^i_{\UU/S}(\log\DD|_{\UU}).
\end{eqnarray*}
Furthermore, we have the Zariski-de Rham complex $(\wt{\Omega}^\bullet_{\XX/S}(\log\DD),d)$.
\end{defn}

\begin{lem}\label{2.4}
With notation and assumptions as above, let $g:T\ra S$ be a morphism of varieties
over $F$, $\YY=\XX\times_S T$, $\VV=\UU\times_S T$, and $\DD'=\DD\times_S T$ with
the induced cartesian diagrams:
\[
\xymatrix{
\VV \ar@{^{(}->}[r]^{\iota'}\ar[d]^{g''} & \YY \ar[r]^{f'}\ar[d]^{g'} & T \ar[d]^g \\
\UU \ar@{^{(}->}[r]^\iota & \XX \ar[r]^f & S.
}
\]
Then for any $i\geq 0$,
\[ g'^*\wt{\Omega}^i_{\XX/S}(\log\DD)\cong\wt{\Omega}^i_{\YY/T}(\log\DD'). \]
\end{lem}

\begin{proof}
Consider the left cartesian diagram. By a direct verification, we can show that
$g'^*\iota_*=\iota'_* g''^*$ holds, hence $g'^*\wt{\Omega}^i_{\XX/S}(\log\DD)
=g'^*\iota_*\Omega^i_{\UU/S}(\log\DD|_{\UU})=\iota'_* g''^*\Omega^i_{\UU/S}(\log\DD|_{\UU})
\cong\iota'_*\Omega^i_{\VV/T}(\log\DD'|_{\VV})=\wt{\Omega}^i_{\YY/T}(\log\DD')$
by \cite[Proposition II.8.10]{ha77}.
\end{proof}

\begin{prop}\label{2.5}
With notation and assumptions as in Lemma \ref{2.4}, assume further that $f:\XX\ra S$
is proper and $S$ is an affine variety.

(i) The sheaves $R^jf_*\wt{\Omega}^i_{\XX/S}(\log\DD)$ and
$\mathbf{R}^nf_*\wt{\Omega}^\bullet_{\XX/S}(\log\DD)$ are coherent.
There exists a nonempty open subset $Z$ of $S$ such that, for any $(i,j)$ and for any $n$,
the restrictions of these sheaves to $Z$ are locally free of finite rank.

(ii) For any $i\in\Z$ and for any morphism $g:T\ra S$, consider the following
cartesian diagram:
\[
\xymatrix{
\YY \ar[r]^{f'}\ar[d]^{g'} & T \ar[d]^g \\
\XX \ar[r]^f & S.
}
\]
Then the base change maps are isomorphisms in $D(T)$:
\begin{eqnarray}
\mathbf{L}g^*\mathbf{R}f_*\wt{\Omega}^i_{\XX/S}(\log\DD) & \longrightarrow &
\mathbf{R}f'_*\wt{\Omega}^i_{\YY/T}(\log\DD'), \label{es4} \\
\mathbf{L}g^*\mathbf{R}f_*\wt{\Omega}^\bullet_{\XX/S}(\log\DD) & \longrightarrow &
\mathbf{R}f'_*\wt{\Omega}^\bullet_{\YY/T}(\log\DD'). \label{es5}
\end{eqnarray}

(iii) Fix $i\in\Z$ and assume that for any $j$, the sheaf $R^jf_*\wt{\Omega}^i_{\XX/S}(\log\DD)$
is locally free over $S$ of constant rank $h^{ij}$. Then for any $j$, the base change map deduced
from (\ref{es4}) is an isomorphism:
\begin{eqnarray}
g^*R^jf_*\wt{\Omega}^i_{\XX/S}(\log\DD) & \longrightarrow &
R^jf'_*\wt{\Omega}^i_{\YY/T}(\log\DD'). \label{es6}
\end{eqnarray}
In particular, $R^jf'_*\wt{\Omega}^i_{\YY/T}(\log\DD')$ is locally free of rank $h^{ij}$.

(iv) Assume that for any $n$, the sheaf $\mathbf{R}^nf_*\wt{\Omega}^\bullet_{\XX/S}(\log\DD)$
is locally free over $S$ of constant rank $h^{n}$. Then for any $n$, the base change map deduced
from (\ref{es5}) is an isomorphism:
\begin{eqnarray}
g^*\mathbf{R}^nf_*\wt{\Omega}^\bullet_{\XX/S}(\log\DD) & \longrightarrow &
\mathbf{R}^nf'_*\wt{\Omega}^\bullet_{\YY/T}(\log\DD'). \label{es7}
\end{eqnarray}
In particular, $\mathbf{R}^nf'_*\wt{\Omega}^\bullet_{\YY/T}(\log\DD')$ is locally free of rank $h^{n}$.
\end{prop}

\begin{proof}
(i) The fact that the sheaves $R^jf_*\wt{\Omega}^i_{\XX/S}$ are coherent is a particular
case of the finiteness theorem of Grothendieck (cf.\ \cite[III.3]{ega} or \cite[Theorem III.8.8]{ha77}
for projective morphisms). The coherence of $\mathbf{R}^nf_*\wt{\Omega}^\bullet_{\XX/S}(\log\DD)$
follows from that of $R^jf_*\wt{\Omega}^i_{\XX/S}$ by the relative Hodge to de Rham spectral sequence.
For the second assertion of (i), denote by $A$ the integral noetherian ring of $S$, $K$ its field of
fractions, which is the localization of $S$ at its generic point $\eta$. Set for brevity
$R^jf_*\wt{\Omega}^i_{\XX/S}=\GG^{ij}$, $\mathbf{R}^nf_*\wt{\Omega}^\bullet_{\XX/S}(\log\DD)=\GG^n$.
The fiber of $\GG^{ij}$ (resp.\ $\GG^n$) at $\eta$ is free of finite rank (i.e.\ a $K$-vector space
of finite dimension), and is the direct limit of $\GG^{ij}|_{D(s)}$ (resp.\ $\GG^n|_{D(s)}$) for
$s$ running through $A$. It follows from Lemma \ref{3.1} that there exists an $s\in A$ such that
$\GG^{ij}|_{D(s)}$ (resp.\ $\GG^n|_{D(s)}$) is free of finite rank.

(ii) Choose a finite covering $\WW$ of $\XX$ by open affine subsets, denote by $\WW'$ the open
covering of $\YY$ deduced from $\WW$ by the base change. Since $\XX$ is proper, hence separated over $S$,
and $S$ is affine, the finite intersections of open subsets in $\WW$ are affine \cite[Ex.II.4.3]{ha77},
and similarly the finite intersections of open subsets in $\WW'$ are affine over $T$.
As a consequence of \cite[Proposition III.8.7]{ha77}, $\mathbf{R}f_*\wt{\Omega}^i_{\XX/S}(\log\DD)$
(resp.\ $\mathbf{R}f'_*\wt{\Omega}^i_{\YY/T}(\log\DD')$) is represented by $f_*\check{\CC}(\WW,
\wt{\Omega}^i_{\XX/S}(\log\DD))$ (resp.\ $f'_*\check{\CC}(\WW',\wt{\Omega}^i_{\YY/T}(\log\DD'))$),
where $\check{\CC}(\WW,\cdot)$ denotes the $\check{\mathrm{C}}$ech complex of sheaves.
By Lemma \ref{2.4}, we have a canonical isomorphism of complexes:
\[
g^*f_*\check{\CC}(\WW,\wt{\Omega}^i_{\XX/S}(\log\DD))\stackrel{\sim}{\longrightarrow}
f'_*\check{\CC}(\WW',\wt{\Omega}^i_{\YY/T}(\log\DD')).
\]
Since the complex $f_*\check{\CC}(\WW,\wt{\Omega}^i_{\XX/S}(\log\DD))$ is bounded and with flat
components, this isomorphism realizes the isomorphism (\ref{es4}).
$\mathbf{R}f_*\wt{\Omega}^\bullet_{\XX/S}(\log\DD)$
(resp.\ $\mathbf{R}f'_*\wt{\Omega}^\bullet_{\YY/T}(\log\DD')$) is similarly represented by
$f_*\check{\CC}(\WW,\wt{\Omega}^\bullet_{\XX/S}(\log\DD))$
(resp.\ $f'_*\check{\CC}(\WW',\wt{\Omega}^\bullet_{\YY/T}(\log\DD'))$),
where $\check{\CC}(\WW,\cdot)$ denotes the associated total complex of the
$\check{\mathrm{C}}$ech bicomplex, and we have a canonical isomorphism of complexes:
\[
g^*f_*\check{\CC}(\WW,\wt{\Omega}^\bullet_{\XX/S}(\log\DD))\stackrel{\sim}{\longrightarrow}
f'_*\check{\CC}(\WW',\wt{\Omega}^\bullet_{\YY/T}(\log\DD')),
\]
which realizes the isomorphism (\ref{es5}).

The conclusions of (iii) and (iv) follow from (ii) and the following lemma,
which is a standard result in homological algebra.
\end{proof}

\begin{lem}\label{2.6}
Let $A$ be a noetherian ring and $E^\bullet$ a bounded complex of finitely generated
projective $A$-modules such that $H^i(E^\bullet)$ are finitely generated
projective $A$-modules for all $i$. Then for any $A$-algebra $B$ and for any $i$,
the canonical homomorphism $B\otimes_A H^i(E^\bullet)\ra H^i(B\otimes_AE^\bullet)$
is an isomorphism.
\end{lem}

\section{The proof of the positive characteristic part}\label{S4}

By Remark \ref{5.2}, almost all of the statements in the positive characteristic
part of the main theorems are indeed proved in \cite{xie}, except that we have
only to remove the simplicial condition from Theorem \ref{1.3} of \cite{xie}.
First of all, we need the following relative vanishing result on toric varieties.

\begin{lem}\label{2.7}
Let $X$ be a projective simplicial toric variety over a perfect field $k$ of
characteristic $p>0$, $Y$ a normal projective variety over $k$, and $f:X\ra Y$
a surjective morphism. Let $H$ be a nef and big torus invariant $\Q$-divisor on $X$.
Then for any $j>0$,
\[ R^jf_*\OO_X(K_X+\ulcorner H\urcorner)=0. \]
\end{lem}

\begin{proof}
Take an ample divisor $L$ on $Y$ and a positive integer $m$ such that
$R^jf_*\OO_X(K_X+\ulcorner H\urcorner)\otimes\OO_Y(mL)$ is generated by its
global sections and $H^i(Y,R^jf_*\OO_X(K_X+\ulcorner H\urcorner)\otimes\OO_Y(mL))=0$
holds for any $i>0$ and any $j\geq 0$. Consider the following Leray spectral sequence:
\[ E_2^{ij}=H^i(Y,R^jf_*\OO_X(K_X+\ulcorner H\urcorner+mf^*L))
\Longrightarrow H^{i+j}(X,K_X+\ulcorner H\urcorner+mf^*L). \]
Since $E_2^{ij}=0$ for any $i>0$, we have $H^0(Y,R^jf_*\OO_X(K_X+\ulcorner H\urcorner+mf^*L))
\cong H^j(X,K_X+\ulcorner H\urcorner+mf^*L)$. Since $H+mf^*L$ is linearly equivalent to a nef
and big torus invariant $\Q$-divisor, we have that $H^j(X,K_X+\ulcorner H\urcorner+mf^*L)=0$
holds for any $j>0$, hence $H^0(Y,R^jf_*\OO_X(K_X+\ulcorner H\urcorner+mf^*L))=0$ holds
for any $j>0$. By global generation, we have $R^jf_*\OO_X(K_X+\ulcorner H\urcorner)\otimes\OO_Y(mL)=0$,
hence $R^jf_*\OO_X(K_X+\ulcorner H\urcorner)=0$ holds for any $j>0$.
\end{proof}

\begin{proof}[Proof of Theorems 1.1--1.3 over a perfect field $k$ of characteristic $p>0$.]
We have only to remove the simplicial condition from Theorem \ref{1.3} of \cite{xie}.

Let $X=X(\Delta,k)$. Take $\pi:X'\ra X$ be an equivariant birational morphism of toric varieties,
where $X'=X(\Delta',k)$, and $\Delta'$ is a subdivision of $\Delta$ by adding some faces of dimension
at least two such that each cone in $\Delta'$ is simplicial. Thus $X'$ is simplicial and $\pi:X'\ra X$
is isomorphic in codimension one.

Let $H'=\pi^*H$. Then $H'$ is also a nef and big torus invariant $\Q$-divisor on $X'$.
Since $\pi:X'\ra X$ is isomorphic in codimension one, we have $\pi_*\OO_{X'}(K_{X'}+\ulcorner H'\urcorner)
=\OO_X(K_X+\ulcorner H\urcorner)$. Consider the following Leray spectral sequence:
\[ E_2^{ij}=H^i(X,R^j\pi_*\OO_{X'}(K_{X'}+\ulcorner H'\urcorner))
\Longrightarrow H^{i+j}(X',K_{X'}+\ulcorner H'\urcorner). \]
Since $E_2^{ij}=0$ for any $j>0$ by Lemma \ref{2.7}, we have $H^i(X,K_X+\ulcorner H\urcorner)
\cong H^i(X',K_{X'}+\ulcorner H'\urcorner)$. By \cite[Theorem 1.3]{xie}, we have
$H^i(X',K_{X'}+\ulcorner H'\urcorner)=0$ for any $i>0$,
hence $H^i(X,K_X+\ulcorner H\urcorner)=0$ holds for any $i>0$.
\end{proof}

The following is a weak version of the relative Kawamata-Viehweg vanishing
on toric varieties in positive characteristic.

\begin{cor}\label{2.8}
Let $X$ be a projective toric variety over a perfect field $k$ of
characteristic $p>0$, $Y$ a normal projective variety over $k$, and $f:X\ra Y$
a surjective morphism. Let $H$ be a nef and big torus invariant $\Q$-divisor on $X$.
Then for any $j>0$,
\[ R^jf_*\OO_X(K_X+\ulcorner H\urcorner)=0. \]
\end{cor}

\begin{proof}
It follows from the positive characteristic part of Theorem \ref{1.3} and
a similar argument to that of Lemma \ref{2.7}.
\end{proof}

\section{The proof of the characteristic zero part}\label{S5}

Almost all of the arguments concerning the reduction modulo $p$ technique, 
except Theorem \ref{3.6}, come from \cite[\S 6]{il}.
Let $\{(A_i)_{i\in I},\,\,u_{ij}:A_i\ra A_j\,(i\leq j)\}$ be a filtered direct system
of rings with direct limit $A$, and denote by $u_i:A_i\ra A$ the canonical homomorphism.
The two most important examples are: (i) a ring $A$ written as a direct limit of its
sub-$\Z$-algebras of finite type; (ii) the localization $A_P$ of a ring $A$ at a prime
ideal $P$ written as a direct limit of localizations $A_f=A[1/f]$ for $f\not\in P$.

Let $\{(E_i)_{i\in I},\,\,v_{ij}:E_i\ra E_j\,(i\leq j)\}$ be a direct system of $A_i$-modules
(resp.\ $A_i$-algebras) with direct limit $E$, where $v_{ij}$ is an $A_i$-linear homomorphism
of $E_i$ into $E_j$ considered as an $A_i$-module (resp.\ $A_i$-algebra) via $u_{ij}$.
We say that $(E_i)_{i\in I}$ is cartesian if, for any $i\leq j$, the natural $A_j$-linear
homomorphism induced by $v_{ij}$
\[ \varphi_{ij}:u_{ij}^*E_i=A_j\otimes_{A_i}E_i\ra E_j,\,\,a_j\otimes e_i\mapsto a_jv_{ij}(e_i) \]
is an isomorphism of $A_j$-modules (resp.\ $A_j$-algebras). In this case, for any $i$,
the natural homomorphism $v_i:E_i\ra E$ induces an isomorphism $\varphi_i:u_i^*E_i
=A\otimes_{A_i}E_i\ra E$ of $A$-modules (resp.\ $A$-algebras).

Let $\{(F_i)_{i\in I},\,\,w_{ij}:F_i\ra F_j\,(i\leq j)\}$ be another direct system of
$A_i$-modules (resp.\ $A_i$-algebras) with direct limit $F$. If $(E_i)$ is cartesian,
then the $\Hom_{A_i}(E_i,F_i)$ form a direct system of $A_i$-modules: the transition
map for $i\leq j$ associated to $f_i:E_i\ra F_i$ is the homomorphism $E_j\ra F_j$
which is the following composition:
\[ E_j\stackrel{\varphi^{-1}_{ij}}{\longrightarrow} A_j\otimes_{A_i}E_i\stackrel{1\otimes f_i}{\longrightarrow}
A_j\otimes_{A_i}F_i\stackrel{\psi_{ij}}{\longrightarrow} F_j, \]
where $\psi_{ij}$ is induced by $w_{ij}$. Analogously, we have maps from $\Hom_{A_i}(E_i,F_i)$
to $\Hom_A(E,F)$, sending $f_i:E_i\ra F_i$ to the homomorphism $E\ra F$ which is
the following composition:
\[ E\stackrel{\varphi^{-1}_i}{\longrightarrow} A\otimes_{A_i}E_i\stackrel{1\otimes f_i}{\longrightarrow}
A\otimes_{A_i}F_i\stackrel{\psi_i}{\longrightarrow} F, \]
where $\psi_i$ is induced by $w_i:F_i\ra F$. Thus we have a natural map:
\begin{eqnarray}
\varinjlim\Hom_{A_i}(E_i,F_i)\ra \Hom_A(E,F). \label{es1}
\end{eqnarray}

Recall that a module (resp.\ an algebra) is said to be finitely presented if it is the
cokernel of a homomorphism between free modules of finite rank (resp.\ free algebras of
finite type). The next lemma shows that hypotheses of finite presentation guarantee that
the direct system $(E_i)$ is cartesian and the map (\ref{es1}) is isomorphic
from certain index $i_0\in I$.

\begin{lem}\label{3.1}
Let $(A_i)_{i\in I}$ be a filtered direct system of rings with direct limit $A$.

(i) If $E$ is a finitely presented $A$-module (resp.\ $A$-algebra), then there exists
$i_0\in I$ and an $A_{i_0}$-module (resp.\ $A_{i_0}$-algebra) $E_{i_0}$ of finite
presentation such that $u_{i_0}^*E_{i_0}\stackrel{\sim}{\longrightarrow} E$.

(ii) Let $(E_i)$ and $(F_i)$ be two direct systems, cartesian for $i\geq i_0$, with direct
limits $E$ and $F$ respectively. If $E_{i_0}$ is finitely presented, then the map
(\ref{es1}) is bijective.
\end{lem}

\begin{rem}\label{3.2}
It follows from Lemma \ref{3.1} that if $E$ is finitely presented, $E_{i_0}$ from which
$E$ arises by extension of scalars is essentially unique, in the sense that if $E_{i_1}$
is another choice (both $E_{i_0}$ and $E_{i_1}$ are of finite presentation), then there
exists $i_2$ with $i_2\geq i_0$ and $i_2\geq i_1$ such that $E_{i_0}$ and $E_{i_1}$ become
isomorphic by extension of scalars to $A_{i_2}$.
\end{rem}

The $S_i=\spec A_i$'s form an inverse system of affine schemes whose inverse limit is
$S=\spec A$. Let $\{(X_i)_{i\in I},\,\,v_{ij}:X_j\ra X_i\}$ be an inverse system of
$S_i$-schemes. We say that this system is cartesian for $i\geq i_0$ if, for any
$i_0\leq i\leq j$, the transition morphism $v_{ij}$ gives a cartesian square:
\[
\xymatrix{
X_j \ar[r]^{v_{ij}}\ar[d] & X_i \ar[d] \\
S_j \ar[r] & S_i.
}
\]
In this case, the $S$-scheme $X$ deduced from $X_{i_0}$ by extension of scalars to
$S$ is the inverse limit of $(X_i)_{i\in I}$. If $(Y_i)_{i\in I}$ is another inverse
system of $S_i$-schemes, cartesian for $i\geq i_0$, with inverse limit
$Y=S\times_{S_{i_0}} Y_{i_0}$, then the $\Hom_{S_i}(X_i,Y_i)$'s form a direct system,
and we have a natural map analogous to (\ref{es1}):
\begin{eqnarray}
\varinjlim\Hom_{S_i}(X_i,Y_i)\ra \Hom_S(X,Y). \label{es2}
\end{eqnarray}

Recall that a morphism of schemes $f:X\ra Y$ is said to be locally of finite presentation
if, there exists a covering of $Y$ by open affine subsets $V_i=\spec B_i$, such that for
each $i$, $f^{-1}(V_i)$ can be covered by open affine subsets $U_{ij}=\spec A_{ij}$,
where each $A_{ij}$ is a finitely presented $B_i$-algebra. If $Y$ is locally noetherian,
then ``locally of finite presentation'' is equivalent to ``locally of finite type''.
Furthermore, $f:X\ra Y$ is said to be of finite presentation if, it is locally of finite
presentation, quasi-compact and quasi-separated. If $Y$ is noetherian, then ``$f$ is of
finite presentation'' is equivalent to ``$f$ is of finite type''.

\begin{prop}\label{3.3}
Let $(S_i)_{i\in I}$ be a filtered inverse system of affine schemes with inverse limit $S$.

(i) If $X$ is an $S$-scheme of finite presentation, then there exists $i_0\in I$ and an
$S_{i_0}$-scheme $X_{i_0}$ of finite presentation from which $X$ is deduced by base change.

(ii) If $(X_i)_{i\in I}$ and $(Y_i)_{i\in I}$ are two inverse systems of $S_i$-schemes,
cartesian for $i\geq i_0$, and if $X_{i_0}$ and $Y_{i_0}$ are finitely presented over
$S_{i_0}$, then the map (\ref{es2}) is bijective.
\end{prop}

\begin{proof}
It is reduced to Lemma \ref{3.1}.
\end{proof}

\begin{rem}\label{3.4}
It follows from Proposition \ref{3.3} that if $X$ is finitely presented over $S$, $X_{i_0}$
from which $X$ arises by base change is essentially unique, in the sense that if $X_{i_1}$
is another choice (both $X_{i_0}$ and $X_{i_1}$ are of finite presentation), then there
exists $i_2$ with $i_2\geq i_0$ and $i_2\geq i_1$ such that $X_{i_0}$ and $X_{i_1}$ become
isomorphic by base change to $S_{i_2}$.
\end{rem}

\begin{prop}\label{3.5}
Let $(S_i)_{i\in I}$ be a filtered inverse system of affine schemes with inverse limit $S$,
and $X$ an $S$-scheme of finite presentation. As in Proposition \ref{3.3}(i), take an $i_0\in I$,
an $S_{i_0}$-scheme $X_{i_0}$ of finite presentation, and the induced cartesian inverse
system $(X_i)$ with $X_i=S_i\times_{S_{i_0}}X_{i_0}$ for $i\geq i_0$.

(i) If $E$ is a finitely presented $\OO_X$-module, then there exists an $i_1\geq i_0$ and an
$\OO_{X_{i_1}}$-module $E_{i_1}$ of finite presentation from which $E$ is deduced by extension
of scalars.

(ii) If $E$ is locally free (resp.\ locally free of rank $r$), then there exists an
$i_2\geq i_1$ such that $E_{i_2}=\OO_{X_{i_2}}\otimes_{\OO_{X_{i_1}}}E_{i_1}$ is locally free
(resp.\ locally free of rank $r$).

(iii) If $X$ is projective over $S$ and $E$ is an ample (resp.\ a very ample) invertible
$\OO_X$-module, then there exists an $i_2\geq i_1$ such that $X_{i_2}$ is projective over
$S_{i_2}$ and $E_{i_2}$ is ample (resp.\ very ample).

(iv) Let $E_{i_0}$ and $F_{i_0}$ be finitely presented $\OO_{X_{i_0}}$-modules, and consider
the cartesian direct systems $(E_i)$ and $(F_i)$ which are deduced by extension of scalars
over the $X_i$'s for $i\geq i_0$, as well as the $\OO_X$-modules $E$ and $F$ which are
deduced by extension of scalars over $X$. Then there is a natural map:
\begin{eqnarray}
\varinjlim_{i\geq i_0}\Hom_{\OO_{X_i}}(E_i,F_i)\ra \Hom_{\OO_X}(E,F), \label{es3}
\end{eqnarray}
which is bijective.
\end{prop}

\begin{proof}
The proofs of (i), (ii) and (iv) are reduced to Lemma \ref{3.1}. For (iii), it suffices to
treat the case where $E$ is very ample, i.e.\ there is a closed immersion $h:X\ra \PP=\PP^r_S$
such that $h^*\OO_\PP(1)\cong E$. By the bijectivity of (\ref{es2}) and (ii), for $i$ sufficiently
large, we can reduce $h$ to an $S_i$-morphism $h_i:X_i\ra \PP_i=\PP^r_{S_i}$ and $E$ to an
invertible module $E_i$ over $X_i$. By the bijectivity of (\ref{es2}) again, up to increasing $i$,
we can obtain that $h_i$ is a closed immersion. By the bijectivity of (\ref{es3}),
up to increasing $i$, we can obtain that the isomorphism $h^*\OO_\PP(1)\cong E$ comes from
an isomorphism $h_i^*\OO_{\PP_i}(1)\cong E_i$. Thus $E_i$ is very ample.
\end{proof}

\begin{thm}\label{3.6}
Let $X$ be a proper toric variety over a field $K$ of characteristic zero, $D$ an
effective divisor, $L$ an ample invertible sheaf and $H$ an ample $\Q$-divisor on $X$.
Then there exists a sub-$\Z$-algebra $A$ of finite type of $K$, a proper toric variety $\XX$
over $S=\spec A$, an effective divisor $\DD$, an ample invertible sheaf $\LL$ and ample
$\Q$-divisor $\HH$ on $\XX$ such that $(X,D,L,H)$ are deduced from $(\XX,\DD,\LL,\HH)$
by base change.
\end{thm}

\begin{proof}
Write $K$ as a direct limit of the family $(A_i)_{i\in I}$ of its sub-$\Z$-algebras
of finite type. Assume that $X=X(\Delta,K)$ is associated to the fan $\Delta$. Then the
properness of $X$ is equivalent to the completeness of $\Delta$. Let $X_i=X(\Delta,A_i)$.
It is easy to see that $(X_i)_{i\in I}$ is a family of proper toric varieties over
$\spec A_i$ from which $X$ is deduced by base change. By Proposition \ref{3.3}, we could
obtain other choices of $X_i$ which may not be toric, however Remark \ref{3.4} shows
that any choice of $X_i$ is essentially unique, i.e.\ $X_i$ must be toric for $i$
sufficiently large.

The effective divisor $D$ is a closed subscheme of $X$, which corresponds to a surjective
homomorphism $\OO_X\ra \OO_D$. By the bijectivity of (\ref{es3}), up to increasing $i$, we can
reduce $\OO_X\ra \OO_D$ to a surjective homomorphism $\OO_{X_i}\ra \OO_{D_i}$ and take $\DD=D_i$.
The reduction of the ample invertible sheaf $L$ to $\LL$ follows from Proposition \ref{3.5}(iii).
For the reduction of the ample $\Q$-divisor $H$, it suffices to treat the case that $H$ is
a very ample divisor. Let $D$ be an effective divisor and $f$ be a rational function on $X$
such that $H=D+\divisor_0(f)$. Take a sufficiently large $i$ such that $A_i$ contains all
coefficients appearing in the rational function $f$. Then up to increasing $i$, we can reduce
$D$ to an effective divisor $\DD$ on $X_i$ and take $\HH=\DD+\divisor_0(f)$.
\end{proof}

We also need the following result on density of closed points.

\begin{prop}\label{3.7}
Let $S$ be a scheme of finite type over $\Z$.

(i) If $x$ is a closed point of $S$, then the residue field $k(x)$ is a finite field.

(ii) Any nonempty locally closed subset $Z$ of $S$ contains a closed point of $S$.
\end{prop}

\begin{proof}
We refer to \cite[IV.10.4.6, IV.10.4.7]{ega}, or in the case where $S$ is affine, this is
a consequence of Hilbert's Nullstellensatz \cite[(14.L)]{ma80}.
\end{proof}

Finally, let us prove the characteristic zero part of the main theorems.

\begin{proof}[Proof of Theorem \ref{1.1}]
Assume that $X=X(\Delta,K)$ is associated to the fan $\Delta$. By Theorem \ref{3.6},
there exists a sub-$\Z$-algebra $A$ of finite type of $K$, a proper toric variety $\XX=X(\Delta,A)$
over $S=\spec A$, a reduced torus invariant Weil divisor $\DD$, and an ample invertible sheaf $\LL$
on $\XX$ such that $(X,D,L)$ are deduced from $(\XX,\DD,\LL)$ by the base change $\iota:\eta\inj S$,
where $\eta=\spec K$ is the generic point of $S$.

Let $\UU=X(\Delta_1,A)$, where $\Delta_1$ is the fan consisting of all 1-dimensional cones in $\Delta$.
Then $\codim_{\XX/S}(\XX-\UU)\geq 2$, and $\UU$ is smooth over $S$. Shrinking $\UU$ if necessary,
we can further assume that $\DD|_\UU$ is relatively simple normal crossing over $S$. Hence by
Definition \ref{2.3}, we obtain the Zariski sheaves $\wt{\Omega}^i_{\XX/S}(\log\DD)$.

Let $f:\XX\ra S$ be the structure morphism. For any $i\geq 0$,
consider the sheaf $\FF=\wt{\Omega}^i_{\XX/S}(\log\DD)\otimes\LL$ on $\XX$.
By the semicontinuity theorem \cite[Theorem III.12.8]{ha77}, the function
$h^j(s,\FF)=\dim_{k(s)}H^j(\XX_s,\FF_s)$ is an upper semicontinuity function
on $S$. In particular, the set $\{s\in S\,|\,h^j(s,\FF)<1\}$ is an open subset
of $S$. By Proposition \ref{3.7}, we can take a closed point $x\in S$ with
the residue field $k(x)$ being a finite field, which is perfect of positive
characteristic, such that $\XX_x=X(\Delta,k(x))$ is a toric variety over $k(x)$.
By the positive characteristic part of Theorem \ref{1.1},
we have $H^j(\XX_x,\FF_x)=0$ for any $j>0$, which implies that the above set
is nonempty, hence is a nonempty open subset of $S$ for any $j>0$. Thus the
generic point $\eta=\spec K$ belongs to this set. As a consequence,
$H^j(\XX_\eta,\FF_\eta)=H^j(X,\wt{\Omega}^i_X(\log D)\otimes L)=0$
holds for any $j>0$.
\end{proof}

\begin{proof}[Proof of Theorem \ref{1.2}]
Set $\dim_K H^j(X,\wt{\Omega}^i_X(\log D))=h^{ij}$, $\dim_K \mathbf{H}^n
(X,\wt{\Omega}^\bullet_X(\log D))=h^n$. It suffices to prove that for any $n$,
we have
\[ h^n=\sum_{i+j=n}h^{ij}. \]

Assume that $X=X(\Delta,K)$ is associated to the fan $\Delta$. By Theorem \ref{3.6},
there exists a sub-$\Z$-algebra $A$ of finite type of $K$, a proper toric variety $\XX=X(\Delta,A)$
over $S=\spec A$, and a reduced torus invariant Weil divisor $\DD$ on $\XX$ such that $(X,D)$ are deduced
from $(\XX,\DD)$ by the base change $\iota:\eta\inj S$, where $\eta=\spec K$ is the generic point of $S$.
By Proposition \ref{2.5}(i), up to replacing $A$ by $A[t^{-1}]$ for a suitable nonzero $t\in A$,
we may assume that the sheaves $R^jf_*\wt{\Omega}^i_{\XX/S}(\log\DD)$ and
$\mathbf{R}^nf_*\wt{\Omega}^\bullet_{\XX/S}(\log\DD)$ are free of constant rank,
which is necessarily equal to $h^{ij}$ and $h^n$ by Proposition \ref{2.5}(iii) and (iv).

By Proposition \ref{3.7}, we can take a closed point $x\in S$ with
the residue field $k(x)$ being a finite field, which is perfect of positive
characteristic, such that $\XX_x=X(\Delta,k(x))$ is a toric variety over $k(x)$.
By the positive characteristic part of Theorem \ref{1.2}, we have for any $n$,
\[
\sum_{i+j=n}\dim_{k(x)}H^j(\XX_x,\wt{\Omega}^i_{\XX_x}(\log \DD_x))=
\dim_{k(x)}\mathbf{H}^{n}(\XX_x,\wt{\Omega}^\bullet_{\XX_x}(\log \DD_x)).
\]
By Proposition \ref{2.5}(iii) and (iv) again, we have for any $i,j$ and for any $n$,
\[
\dim_{k(x)}H^j(\XX_x,\wt{\Omega}^i_{\XX_x}(\log \DD_x))=h^{ij},\,\,
\dim_{k(x)}\mathbf{H}^{n}(\XX_x,\wt{\Omega}^\bullet_{\XX_x}(\log \DD_x))=h^n.
\]
Therefore $h^n=\sum_{i+j=n}h^{ij}$ for any $n$, which implies the conclusion.
\end{proof}

\begin{proof}[Proof of Theorem \ref{1.3}]
By Kodaira's Lemma, we can take an effective torus invariant $\Q$-divisor $B$ with sufficiently small
coefficients such that $H-B$ is ample and $\ulcorner H-B\urcorner=\ulcorner H\urcorner$.
Therefore, we may from the beginning assume that $H$ is an ample torus invariant $\Q$-divisor on $X$.

Assume that $X=X(\Delta,K)$ is associated to the fan $\Delta$. By Theorem \ref{3.6},
there exists a sub-$\Z$-algebra $A$ of finite type of $K$, a proper toric variety $\XX=X(\Delta,A)$
over $S=\spec A$, and an ample $\Q$-divisor $\HH$ on $\XX$ such that $(X,H)$ are deduced from
$(\XX,\HH)$ by the base change $\iota:\eta\inj S$, where $\eta=\spec K$ is the generic point of $S$.

Let $f:\XX\ra S$ be the structure morphism. Consider the sheaf $\FF=\wt{\Omega}^n_{\XX/S}\otimes
\OO_{\XX}(\ulcorner\HH\urcorner)$ on $\XX$. By the semicontinuity theorem \cite[Theorem III.12.8]{ha77},
the function $h^j(s,\FF)=\dim_{k(s)}H^j(\XX_s,\FF_s)$ is an upper semicontinuity function
on $S$. In particular, the set $\{s\in S\,|\,h^j(s,\FF)<1\}$ is an open subset
of $S$. By Proposition \ref{3.7}, we can take a closed point $x\in S$ with
the residue field $k(x)$ being a finite field, which is perfect of positive
characteristic, such that $\XX_x=X(\Delta,k(x))$ is a toric variety over $k(x)$.
By the positive characteristic part of Theorem \ref{1.3},
we have $H^j(\XX_x,\FF_x)=0$ for any $j>0$, which implies that the above set
is nonempty, hence is a nonempty open subset of $S$ for any $j>0$. Thus the
generic point $\eta=\spec K$ belongs to this set. As a consequence,
$H^j(\XX_\eta,\FF_\eta)=H^j(X,K_X+\ulcorner H\urcorner)=0$
holds for any $j>0$.
\end{proof}

\textsc{School of Mathematical Sciences, Fudan University,
Shanghai 200433, China}

\textit{E-mail address}: \texttt{qhxie@fudan.edu.cn}


\small
\begin{thebibliography}{BTLM97}

\bibitem[BTLM97]{btlm}
A. Buch, J. F. Thomsen, N. Lauritzen, V. Mehta,
The Frobenius morphism on a toric variety,
{\it Tohoku Math.\ J.}, {\bf 49} (1997), 355--366.

\bibitem[CLH11]{clh}
D. Cox, J. Little, H. Schenck,
{\it Toric varieties}, Graduate Studies in Mathematics, vol.\ {\bf 124},
American Mathematical Society, 2011.

\bibitem[Da78]{da}
V. I. Danilov,
The geometry of toric varieties,
{\it Russ.\ Math.\ Surv.}, {\bf 33} (1978), 97--154.

\bibitem[DI87]{di}
P. Deligne, L. Illusie,
Rel\`{e}vements modulo $p^2$ et d\'{e}composition du complexe de de Rham,
{\it Invent.\ Math.}, {\bf 89} (1987), 247--270.

\bibitem[EGA]{ega}
A. Grothendieck, J. Dieudonn\'e,
{\it El\'ements de G\'eom\'etrie Alg\'ebrique},
Publ.\ Math.\ IHES.

\bibitem[Fu07]{fu07}
O. Fujino,
Multiplication maps and vanishing theorems for toric varieties,
{\it Math.\ Zeit.}, {\bf 257} (2007), 631--641.

\bibitem[Fu93]{fu93}
W. Fulton,
{\it Introduction to toric varieties}, Ann.\ of Math.\ Stud.,
vol.\ {\bf 131}, Princeton Univ.\ Press, 1993.

\bibitem[Ha77]{ha77}
R. Hartshorne, {\it Algebraic geometry}, Springer-Verlag, 1977.

\bibitem[Il96]{il}
L. Illusie,
Frobenius et d\'eg\'en\'erescence de Hodge,
in J. Bertin, J.-P. Demailly, L. Illusie and C. Peters,
{\it Introduction \`a la Th\'eorie de Hodge},
Panoramas et Synth\`eses, vol.\ {\bf 3},
Socit\'e de Mathmatiques de France, Marseilles,
113-168, 1996.

\bibitem[Ma80]{ma80}
H. Matsumura,
{\it Commutative Algebra}, Benjamin, 1980.

\bibitem[Mo82]{mo}
S. Mori,
Threefolds whose canonical bundles are not numerically effective,
{\it Ann.\ Math.}, {\bf 116} (1982), 133--176.

\bibitem[Mu02]{mu}
M. Musta\c{t}\v{a},
Vanishing theorems on toric varieties,
{\it Tohoku Math.\ J.}, {\bf 54} (2002), 451--470.

\bibitem[Od88]{od}
T. Oda,
{\it Convex bodies and algebraic geometry}, Springer-Verlag, 1988.

\bibitem[Xie14]{xie}
Q. Xie,
Vanishing theorems on toric varieties in positive characteristic,
{\it Math.\ Zeit.}, {\bf 276} (2014), 191--202.

\end{thebibliography}
\end{document}